\documentclass[11pt,a4paper,reqno]{amsart}

\usepackage[utf8]{inputenc}
\usepackage[a4paper, lmargin=0.12\paperwidth, rmargin=0.12\paperwidth, tmargin=0.11\paperheight, bmargin=0.11\paperheight]{geometry}
\usepackage{amsfonts,amsmath,amssymb,amsthm,dsfont,amsxtra,enumerate,multicol,mathtools,tkz-berge,hyperref,graphicx}

\newtheorem{theorem}{Theorem}[section]
\newtheorem{corollary}[theorem]{Corollary}
\newtheorem{lemma}[theorem]{Lemma}
\newtheorem{proposition}[theorem]{Proposition}

\newtheorem{definition}[theorem]{Definition}
\theoremstyle{definition}
\newtheorem{remark}[theorem]{Remark}
\newtheorem{example}[theorem]{Example}

\newtheorem{notation}[theorem]{Notation}
\numberwithin{equation}{section}

\DeclareMathOperator{\reg}{reg}

\DeclareMathOperator{\ini}{in}

\DeclareMathOperator{\hp}{P}

\newcommand{\mcC}{\mathcal{C}}

\title{Ideals generated by corner-interval minors}
\author[Marie Amalore Nambi]{Marie Amalore Nambi}
\address{Sabanci University, Faculty of Engineering and Natural Sciences, Orta Mahalle, Tuzla, 34956, Istanbul, Turkey}
\email{amalore.p@gmail.com, amalore.pushparaj@sabanciuniv.edu}


\date{}

\begin{document}

\begin{abstract}
     In this article, we study binomial ideals generated by an arbitrary collection of corner-interval $2$-minors of a generic matrix. We determine the minimal prime ideals of such ideals and characterize their radicality in the special case of corner minors. Moreover, we discuss connectivity properties of contingency tables in algebraic statistics. We compute the Hilbert-Poincar\'e polynomial of the ideal generated by the set of all corner-interval minors and we derive the formula for the regularity in the case of corner minors. 
\end{abstract}

\maketitle

\section{Introduction}

Let $M=(x_{ij})_{i=1,\ldots,m \atop j=1,\ldots,n}$ be an $m \times n$ matrix of indeterminates over a field $K$. The ideals generated by all $r$-minors of $M$, known as determinantal ideals of $M$, have been extensively studied from various perspectives, see the lecture notes and survey \cite{BV1988, BC2003}. In recent past, motivated by applications in algebraic statistics, more specifically in the study of conditional independence ideals and the connectedness of contingency tables, researchers have been interested in the study of ideals generated by arbitrary subsets of $2$-minors of $M$, see \cite{DES1998, HHHKR2010}.

  \vspace{2mm}

In combinatorial commutative algebra, when $2$-minors are associated with combinatorial objects, researchers are interested in studying the algebraic properties of the corresponding ideals through their underlying combinatorial structures. In this context, Herzog et al. \cite{HHHKR2010} and independently Ohtani \cite{O2011} introduced the notion of binomial edge ideal corresponding to a finite simple graph. The binomial edge ideal of a graph can be visualized as the ideal generated by some set of $2$-minors of a $2\times n$ matrix. The authors showed that the binomial edge ideals of graphs are radical. Moreover, the authors characterized its primary decomposition. However, even when the matrix is extended to size $3 \times n$, our understanding is limited, even a simple example shows that there are not radical.

  \vspace{2mm}

Motivated by these, in this paper we focus on the ideal generated by $2$-minors that arise from an arbitrary subset of corner-interval minors.  A corner-interval minor of $M$ is a $2$-minor of the form $[a_1,a_2|1,b_2]$, where $a_1,a_2 \in [m]$ and $b_2 \in [n]$. This includes corner minors, a class of $2$-minors introduced by Diaconis, Eisenbud, and Sturmfels in \cite{DES1998}. This class of ideals can be viewed as a natural generalization of the generalized binomial edge ideal of a star graph.

  \vspace{2mm}

Rauh \cite{R2013} introduced the generalized binomial edge ideal of graphs, while Ene et al. \cite{EHHQ2014} extended it by introducing the binomial edge ideal of a pair of graphs. Both of these classes of ideals are natural extensions of binomial edge ideals, generalizing the setting from a $2 \times n$ matrix to an arbitrary $m \times n$ matrix. Moreover, the authors studied Gr\"obner basis and characterized radical and minimal primes. The significance of studying the generalized binomial edge ideal of graphs lies in its applications to algebraic statistics, see \cite{R2013}. The ideal generated by adjacent $2$-minors of a matrix $M$ was introduced in \cite{DES1998} and their algebraic properties such as prime, radical and minimal prime ideals are studied in \cite{HS2004, HH2012}. The ideal generated by all adjacent $2$-minors of $M$ forms a lattice basis ideal. Moreover, the lattice ideal associated with this basis is the ideal generated by all $2$-minors of $M$ (cf. \cite{ES1996}). The class of ideals of diagonal $2$-minors of a $n \times n$ matrix is studied in \cite{EQ2013}.  Qureshi \cite{Q2012} associated a combinatorial object known as a polyomino with the ideal generated by its inner $2$-minors, referring to the resulting ideal as a polyomino ideal. The author proved that the ideals associated with convex polyominoes are prime. The characterization of the primality and radicality of the polyomino ideals is still open. The primary objective of this article is to study the minimal prime ideals and radicality of ideals generated by arbitrary collections of corner-interval minors of $M$.

\vspace{2mm}

 Let $\mcC$ be an arbitrary subset of corner-interval minors. The ideal  $$I(\mcC) =(x_{a_11}x_{a_2b_2}-x_{a_1b_2}x_{a_21} \mid \text{ for every } [a_1,a_2|1,b_2] \in \mathcal{C}) \subset S=K[x_{ij} \mid x_{ij}\in V(\mcC)]$$ is called the binomial ideal of $\mcC$. One effective approach to studying the primality of the ideal $I(\mcC)$ is through its relationship with a toric ideal. In Section \ref{sec.min},  we construct a bipartite graph associated with $\mcC$ and compare the ideal $I(\mcC)$ with the toric ideal of the edge ring of this graph. It turns out that the toric ideal is a minimal prime of $I(\mcC)$ that does not contain any variables, see Lemma \ref{lemma.primenovar}. Moreover, we give an explicit description of the minimal prime ideals of $I(\mcC)$ by introducing admissible sets, a purely combinatorial way, see Theorem \ref{thm.minimal}. In \cite{ES1996}, the authors obtained a formula for the primary decomposition of the ideal generated by the set of all corner minors of $M$ (see \cite[Theorem 3.2]{ES1996}). This result motivates us to study the special case in which $\mcC$ is an arbitrary collection of corner minors, which we study in Section \ref{sec.cor}. We provide a characterization of the radicality of the ideal $I(\mcC)$, see Corollary \ref{thm.rad}. Moreover, we obtain a decomposition of the ideal $I(\mcC)$, as shown in Proposition \ref{prop.decomp}. In the case where the ideal is radical, we further describe its primary decomposition. As an application, we describe when contingency tables are connected via corner minors of $\mcC$.

\vspace{2mm}

The second part of this article is dedicated to the study of the Hilbert-Poincar\'e series of the generalized binomial edge ideal of star graphs, as well as the Castelnuovo-Mumford regularity of the binomial edge ideal of a pair of star graphs. Several authors have found exact formulas or bounds for the regularity of the generalized binomial edge ideal of graphs, see \cite{AMS24, CI21, K2020, SZ25,  SK13}. However, the regularity of the binomial edge ideals of pair of graphs are not that much explored. In Section \ref{sec.bet}, we obtain the regularity of the binomial ideal generated by the set of all corner minors of $M$, see Theorem \ref{thm.regc}. In addition, we compute the Hilbert-Poincar\'e polynomial of the generalized binomial edge ideal of a star graph, see Theorem \ref{thm.hpp}. Finally, we obtain a lower bound for the regularity of powers of the ideal generated by an arbitrary set of $2$-minors of corner-intervals of $M$ in terms of its subcorner-interval minors.

\section{Preliminaries} \label{sec.pre}

In this section, we recall some notations, definitions, and results from graph theory and commutative algebra that will be used throughout the article.

\vspace{2mm}

Let $R=K[x_{ij} \mid i \in [m], j \in [n]]$ be a polynomial ring in $mn$ indeterminate over a field $K$. Let $M=(x_{ij})_{i=1,\ldots,m \atop j=1,\ldots,n}$ be an $m\times n$ matrix of indeterminates. For a $2$-minor $\delta = [a_1,a_2\mid b_1,b_2]$ of $M$, the variables $x_{a_1b_1},x_{a_1b_2},x_{a_2b_1}$, and $x_{a_2b_2}$ are called vertices of $\delta$, denoted by $V(\delta)$ and the sets $\{x_{a_1b_1},x_{a_1b_2}\}, \{x_{a_1b_1},x_{a_2b_1}\}$, $\{x_{a_1b_2},x_{a_2b_2}\}$, and $\{x_{a_2b_1},x_{a_2b_2}\}$ are called edges of the $2$-minor $\delta$. Let $\mcC$ be an arbitrary set of $2$-minors of $M$. The vertex set of $\mathcal{C}$, denoted by $V(\mathcal{C})$, is defined as the union of $V(\delta)$ for every $\delta$ in $\mathcal{C}$. The edge set of $\mathcal{C}$, denoted by $E(\mathcal{C})$, is the union of the edge sets of all $\delta$ in $\mathcal{C}$. For a subset $U\subset V(\mcC)$, we denote $\mcC \setminus U$ the collection of minors in $\mcC$ whose vertex sets are disjoint from $U$; that is, $\mcC\setminus U=\{\delta \in \mcC \mid V(\delta) \cap U = \emptyset\}$.

Let $\mathcal{C}$ be an arbitrary set of $2$-minors of $M$. For every $\delta=[a_1,a_2|b_1,b_2] \in \mathcal{C}$ we associate a binomial $f_{\delta}=x_{a_1b_1}x_{a_2b_2}-x_{a_1b_2}x_{a_2b_1}$. Then the ideal $$I(\mathcal{C})=(f_{\delta} \mid \delta \in \mathcal{C}) \subset S=K[x_{ij} \mid x_{ij} \in V(\mcC)]$$ is called binomial ideal of $\mcC$. 

In the following, we recall the definition of corner minor of $M$, as introduced in \cite{DES1998}. 

\begin{definition}   (cf. \cite{DES1998})
    The $2$-minor $[a_1,a_2|b_1,b_2]$ of $M$ is called corner minor if $a_1=1$ and $b_1=1$.
\end{definition}
\begin{definition} 
    The $2$-minor $[a_1,a_2|b_1,b_2]$ of $M$ is called corner-interval minor if $b_1=1$. 
\end{definition}

\subsection{Graphs}
Let $G=(V(G) = [n],E(G))$ be a simple graph, where $[n]=\{1,\ldots,n\}$. A graph is called complete if $E(G)=\{ \{i,j\} \mid 1 \leq i <j\leq n\}$. A graph is said to be a cycle if $E(G)=\{ \{1,n\}\} \cup \{\{i,i+1\} \mid 1 \leq i \leq n-1\}$, for $n \geq 3$. A graph is called chordal if every cycle of length at least four has a chord. A graph is said to be cochordal if its complement is chordal. A graph $G$ is called bipartite if the vertex set $V(G)$ can be partitioned as $V(G)=V_1 \sqcup V_2$, and for every edge $\{u_1,u_2\} \in E(G)$, one has $u_1 \in V_1$ and $u_2 \in V_2$.  For $m,n >0$, $K_{m,n}$ denotes the \textit{complete bipartite} graph on $[m+n]$. The graph $K_{1,m}$ is called the star graph. We recall that the edge ideal of $G$ over a field $K$, denoted by $I(G)$, is defined by $$I(G)=(x_ix_j \mid \{i,j\}\in E(G)) \subset K[x_1,\ldots,x_n].$$

Next, we recall the definition of the binomial edge ideal of pair of graphs (cf. \cite{EHHQ2014}) and its connection with the binomial ideal of $\mcC$, where $\mcC$ is an arbitrary collection of corner-interval minors of $M$.

Let $G_1$ be a simple graph on $[m]$ and $G_2$ be a simple graph on $[n]$. Let $e=\{i,j\}$ for some $1\leq i < j \leq m$ and $f=\{k,l\}$ for some $1\leq k < l \leq n$.  For every $e\in E(G_1)$ and $ f \in E(G_2)$ we associate a binomial $p_{e,f}=x_{ik}x_{jl}-x_{il}x_{jk}$. The ideal 
$$J_{G_1,G_2}=(p_{e,f} \mid e\in E(G_1), f \in E(G_2)) \subset S$$
is called the binomial edge ideal of pair $(G_1,G_2)$. 

\begin{example}
\begin{enumerate}
    \item Let $\mcC$ be the set of all corner minors of $M$. Then the ideal $I(\mcC)$ coincides with the ideal $J_{G_1,G_2}$, where $G_1$ and $G_2$ are star graphs. 
    \item Let $\mcC$ be the set of all corner-interval minors of the matrix $M$. Then the binomial ideal $I(\mcC)$ coincides with the ideal $J_{G_1,G_2}$, where $G_1$ is the complete graph on $[m]$ and $G_2$ is the star graph on $[n]$. 
\end{enumerate}
\end{example}

\begin{remark}\cite[Proposition 5.4]{S1995} \label{rem.Mgb}
    If $G_1$ and $G_2$ are complete graphs, then $J_{G_1,G_2} = I_2(M)$ and the minimal generating set of $J_{G_1,G_2}$ forms a reduced Gr\"obner basis under a suitable term order. 
\end{remark}

\subsection{Commutative algebra}
Let $S=K[X_1,\ldots,X_n]$ be a standard graded polynomial ring in $n$ variables over a field $K$. Let $M$ be a finitely generated graded $S$-module. Let $\mathbf{F}_{\bullet}$ be a minimal graded free $S$-resolution of $M$:
$$\mathbf{F}_{\bullet}: 0 \longrightarrow \bigoplus_{j}S(-j)^{\beta_{p,j}(M)}   \longrightarrow  \bigoplus_{j}S(-j)^{\beta_{p-1,j}(M)} \longrightarrow  \cdots \longrightarrow  \bigoplus_{j}S(-j)^{\beta_{0,j}(M)}   \longrightarrow M \longrightarrow  0,$$
where $S(-j)$ denotes the graded free module of rank $1$ obtained by shifting the degrees in $S$ by $j$, and $\beta_{i,j}(M)$ denotes the $(i,j)$-th graded Betti number of $M$ over $S$. 

The \emph{Castelnuovo-Mumford regularity} of $M$ over $S$, denoted by $\reg (M)$, is defined as $$\reg(M) \coloneqq \max \{j-i \mid \beta_{i,j}(M) \neq 0\}. $$

We recall the definition of Hilbert series as follows (cf. \cite[Chapter 6]{HH2011}). Let $H(M,i)$ be the Hilbert function of $M$. The Hilbert-Poincare series of $S$-module $M$ is $$H_M(z)=\sum_{i\geq 0} H(M,i)z^i.$$
By the Hilbert–Serre theorem, the Hilbert-Poincare series of $M$ is a rational function, we have $$H_M(z)= \hp_M(z)/(1-z)^n.$$ The numerator is the Hilbert-Poincar\'e polynomial of $M$ and has the form $$\hp_M(z)=\sum_{i=0}^{p}\sum_{j=0}^{p+r}(-1)^i\beta_{i,j}(M)z^j,$$ where $p$ is the projective dimension of $M$ and $r$ is the regularity of $M$.

\vspace{2mm}

In the following remark we recall a basic property of the Hilbert-Poincar\'e polynomial and the regularity of an $S$-module.

\begin{remark}\label{rem.Reg} \cite[Lemma 3.1]{HT2010} 
Let $ 0 \rightarrow \mathcal{L} \rightarrow \mathcal{M} \rightarrow \mathcal{N}   \rightarrow  0$ be a short exact sequence of finitely generated graded $S$-modules. Then the following holds.
\begin{enumerate}[(a)]
    \item $\hp_{\mathcal{M}}(z)=\hp_{\mathcal{L}}(z)+\hp_{\mathcal{N}}(z)$.
   \item $\reg{\mathcal{L}} \leq \max \{\reg \mathcal{M}, \reg \mathcal{N} +1\}$. The equality holds if $\reg \mathcal{M} \neq \reg \mathcal{N}$.
   
   \item $\reg{\mathcal{M}} \leq \max \{\reg \mathcal{L}, \reg \mathcal{N}\}$. The equality holds if $\reg \mathcal{L} \neq \reg \mathcal{N} +1$.

\end{enumerate}
\end{remark}

\begin{remark} \cite[Exercise 15.4]{E1995} \label{rem.starhp}
    Let $K_{1,m}$ be a star graph. Then one has
    $$\hp_{S/I(K_{1,m})}(z)=z(1-z)^{m}+(1-z).$$
\end{remark}

The following result characterizes the edge ideals of finite simple graphs with linear resolution.

\begin{remark}\cite[Theorem 1]{F1990}\label{rem.froberg}
    Let $G$ be a finite simple graph. Then $I(G)$ has regularity $2$ if and only if the complement graph of $G$ is chordal.
\end{remark}

In the following, we recall a useful combinatorial criterion to address the membership problem, as given in \cite{DES1998}.

    Let $\mathcal{B}$ be a set of vectors in $\mathbb{Z}^n$. Define a graph $G_{\mathcal{B}}$ whose vertices are non negative $n$-tuples $\mathbb{N}^n$ and $\{u,v\} \in E(G_{\mathcal{B}})$ if and only if $u-v$ is in $\pm \mathcal{B}$. For every vector $u=(u_1,\ldots,u_n)$ of positive integers, we define a monomial $x^u=x_1^{u_1}x_2^{u_2}\cdots x_n^{u_n} \in K[x_1,\ldots,x_n]$. Every vector $u\in \mathbb{Z}^n$ can be uniquely written as $u=u_+-u_-$, where $u_+,u_-$ are non-negative vectors with disjoint support. To the subset $\mathcal{B}$ we can associate an ideal as 
$$I_{\mathcal{B}}=(x^{u_+}-x^{u_-} \mid u \in \mathcal{B}) \subset K[x_1,\ldots,x_n].$$

\begin{remark}\cite[Theorem 1.1]{DES1998} \label{rem.bwalk}
    Two vectors $u,v \in \mathbb{N}^n$ are in the same component of $G_{\mathcal{B}}$ if and only if $x^u-x^v \in I_{\mathcal{B}}$.
\end{remark}

For convenience, we denote $\mathcal{B_C}$ for the basis vectors corresponding to the set $\mcC$.

\section{Minimal prime ideals} \label{sec.min}

In this section, we study binomial ideals generated by corner-interval minors and aim to explicitly describe their minimal prime ideals. We begin by introducing the necessary definitions.

\vspace{2mm}

A subset $U=\{x_{i_1j},\ldots,x_{i_k,
j}\} \subset V(\mathcal{C})$ is called \textit{vertical interval} if every pair of vertices in $U$ is connected by a sequence of vertical edges (path) of $\mathcal{C}$. We call a  vertical interval \textit{maximal} if it is not strictly contained in any other vertical interval of $\mathcal{C}$. Similarly, one can define horizontal interval and maximal horizontal interval.

\begin{example}
    Let $\mcC$ be the corner-interval minors as shown in Figure \ref{fig.maxC}. Then the ideal 
    $$I(\mcC) = (x_{12}x_{21}-x_{11}x_{22},x_{23}x_{31}-x_{21}x_{33},x_{14}x_{31}-x_{11}x_{34},x_{32}x_{41}-x_{31}x_{42}).$$ 
    The right side of Figure \ref{fig.maxC} shows the corner-interval $\mcC$ with the maximal vertical intervals $V_1,V_2,V_3,V_4,V_5$ and the maximal horizontal intervals $H_1,H_2,H_3,H_4$. The set $\{x_{12},x_{22},x_{32}\}$ does not form a vertical interval of $\mcC$, as the vertex $x_{32}$ is not connected with either $x_{12}$ or $x_{22}$ by any sequence of vertical edges of $\mcC$. 
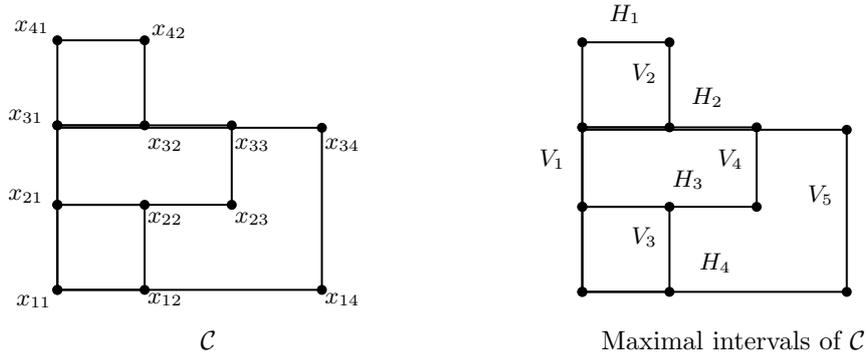
\begin{figure}[ht]
    \centering
    \tikzset{every picture/.style={line width=0.75pt}} %set default line width to 0.75pt        

\begin{tikzpicture}[x=0.75pt,y=0.75pt,yscale=-1,xscale=1]
%uncomment if require: \path (0,941); %set diagram left start at 0, and has height of 941

%Shape: Rectangle [id:dp02597992884704059] 
\draw   (210,397.9) -- (253.54,397.9) -- (253.54,440.68) -- (210,440.68) -- cycle ;
%Shape: Rectangle [id:dp08446933465928974] 
\draw   (210,315.15) -- (253.54,315.15) -- (253.54,357.93) -- (210,357.93) -- cycle ;
%Shape: Rectangle [id:dp6604549549607484] 
\draw   (210,357.93) -- (297.03,357.93) -- (297.03,397.9) -- (210,397.9) -- cycle ;
%Shape: Rectangle [id:dp48079625204031795] 
\draw   (210,359.15) -- (342.03,359.15) -- (342.03,440.68) -- (210,440.68) -- cycle ;
%Shape: Rectangle [id:dp82414316980454] 
\draw   (472,398.9) -- (515.54,398.9) -- (515.54,441.68) -- (472,441.68) -- cycle ;
%Shape: Rectangle [id:dp9793117046860698] 
\draw   (472,316.15) -- (515.54,316.15) -- (515.54,358.93) -- (472,358.93) -- cycle ;
%Shape: Rectangle [id:dp17731448843822595] 
\draw   (472,358.93) -- (559.03,358.93) -- (559.03,398.9) -- (472,398.9) -- cycle ;
%Shape: Rectangle [id:dp19853309674922826] 
\draw   (472,360.15) -- (604.03,360.15) -- (604.03,441.68) -- (472,441.68) -- cycle ;

% Text Node
\draw (188,441.68) node [anchor=north west][inner sep=0.75pt]  [font=\footnotesize]  {$x_{11}$};
% Text Node
\draw (253.54,440.68) node [anchor=north west][inner sep=0.75pt]  [font=\footnotesize]  {$x_{12}$};
% Text Node
\draw (342.03,440.68) node [anchor=north west][inner sep=0.75pt]  [font=\footnotesize]  {$x_{14}$};
% Text Node
\draw (184,388.68) node [anchor=north west][inner sep=0.75pt]  [font=\footnotesize]  {$x_{21}$};
% Text Node
\draw (184,348.68) node [anchor=north west][inner sep=0.75pt]  [font=\footnotesize]  {$x_{31}$};
% Text Node
\draw (187,303.68) node [anchor=north west][inner sep=0.75pt]  [font=\footnotesize]  {$x_{41}$};
% Text Node
\draw (255.54,306.15) node [anchor=north west][inner sep=0.75pt]  [font=\footnotesize]  {$x_{42}$};
% Text Node
\draw (253.54,363) node [anchor=north west][inner sep=0.75pt]  [font=\footnotesize]  {$x_{32}$};
% Text Node
\draw (297.03,400) node [anchor=north west][inner sep=0.75pt]  [font=\footnotesize]  {$x_{23}$};
% Text Node
\draw (297.03,363) node [anchor=north west][inner sep=0.75pt]  [font=\footnotesize]  {$x_{33}$};
% Text Node
\draw (342.03,363) node [anchor=north west][inner sep=0.75pt]  [font=\footnotesize]  {$x_{34}$};
% Text Node
\draw (253.54,400) node [anchor=north west][inner sep=0.75pt]  [font=\footnotesize]  {$x_{22}$};
% Text Node
\draw (484,295.68) node [anchor=north west][inner sep=0.75pt]  [font=\footnotesize]  {$H_{1}$};
% Text Node
\draw (529,420.68) node [anchor=north west][inner sep=0.75pt]  [font=\footnotesize]  {$H_{4}$};
% Text Node
\draw (525,337.68) node [anchor=north west][inner sep=0.75pt]  [font=\footnotesize]  {$H_{2}$};
% Text Node
\draw (515.52,378.92) node [anchor=north west][inner sep=0.75pt]  [font=\footnotesize]  {$H_{3}$};
% Text Node
\draw (583,386.68) node [anchor=north west][inner sep=0.75pt]  [font=\footnotesize]  {$V_{5}$};
% Text Node
\draw (449,368.68) node [anchor=north west][inner sep=0.75pt]  [font=\footnotesize]  {$V_{1}$};
% Text Node
\draw (495,325.68) node [anchor=north west][inner sep=0.75pt]  [font=\footnotesize]  {$V_{2}$};
% Text Node
\draw (495,407.68) node [anchor=north west][inner sep=0.75pt]  [font=\footnotesize]  {$V_{3}$};
% Text Node
\draw (538,368.68) node [anchor=north west][inner sep=0.75pt]  [font=\footnotesize]  {$V_{4}$};

\draw (480,460) node [anchor=north west][inner sep=0.75pt]  [font=\small]  {Maximal intervals of $\mcC$};
\draw (280,460) node [anchor=north west][inner sep=0.75pt]  [font=\small]  {$\mcC$};

%\filldraw[black] (472,360.15) circle (1.5pt) ;
\filldraw[black] (604.03,360.15) circle (1.5pt) ;
\filldraw[black] (604.03,441.68) circle (1.5pt) ;
\filldraw[black] (472,441.68) circle (1.5pt) ;

\filldraw[black] (472,358.93)  circle (1.5pt) ;
\filldraw[black] (559.03,358.93) circle (1.5pt) ;
\filldraw[black] (559.03,398.9) circle (1.5pt) ;
\filldraw[black] (472,398.9) circle (1.5pt) ;

\filldraw[black]  (472,398.9) circle (1.5pt) ;
\filldraw[black] (515.54,398.9) circle (1.5pt) ;
\filldraw[black] (515.54,441.68) circle (1.5pt) ;
\filldraw[black] (472,441.68) circle (1.5pt) ;

\filldraw[black] (472,316.15)  circle (1.5pt) ;
\filldraw[black] (515.54,316.15) circle (1.5pt) ;
\filldraw[black] (515.54,358.93) circle (1.5pt) ;
\filldraw[black] (472,358.93) circle (1.5pt) ;

\filldraw[black] (210,397.9)  circle (1.5pt) ;
\filldraw[black] (253.54,397.9)  circle (1.5pt) ;
\filldraw[black] (253.54,440.68)  circle (1.5pt) ;
\filldraw[black] (210,440.68)  circle (1.5pt) ;

\filldraw[black] (210,315.15)  circle (1.5pt) ;
\filldraw[black] (253.54,315.15)  circle (1.5pt) ;
\filldraw[black]  (253.54,357.93) circle (1.5pt) ;
\filldraw[black] (210,357.93)  circle (1.5pt) ;

\filldraw[black] (210,357.93)  circle (1.5pt) ;
\filldraw[black] (297.03,357.93)  circle (1.5pt) ;
\filldraw[black] (297.03,397.9) circle (1.5pt) ;
\filldraw[black] (210,397.9) circle (1.5pt) ;

%\filldraw[black] (210,359.15)  circle (1.5pt) ;
\filldraw[black] (342.03,359.15)  circle (1.5pt) ;
\filldraw[black] (342.03,440.68) circle (1.5pt) ;
\filldraw[black] (210,440.68) circle (1.5pt) ;
\end{tikzpicture}

    \caption{The corner-interval $\mcC$ and the maximal intervals of $\mcC$.}
    \label{fig.maxC}
\end{figure}
\end{example}

Let $V_1,\ldots,V_p$ be the set of  maximal vertical intervals and $H_1,\ldots,H_q$ be the set of maximal horizontal intervals of $\mathcal{C}$. The associated bipartite graph of $\mathcal{C}$, denoted by $G(\mathcal{C})$, is defined as with vertex set $V(G(\mathcal{C}))=\{v_1,\ldots,v_p\}\sqcup \{h_1,\ldots,h_q\}$, and the edge set $$E(G(\mathcal{C}))= \{\{h_i,v_j\} \mid H_i \cap V_j \neq \emptyset \}.$$

\begin{example}
     Let $\mcC$ be the corner-interval minors as shown in Figure \ref{fig.maxC}. The associated bipartite graph $G(\mcC)$ of $\mcC$ is depicted in  Figure \ref{fig:g(c)}.

\begin{figure}[ht]
    \centering
    \tikzset{every picture/.style={line width=0.75pt}} %set default line width to 0.75pt        

\begin{tikzpicture}[x=0.75pt,y=0.75pt,yscale=-1,xscale=1]
%uncomment if require: \path (0,941); %set diagram left start at 0, and has height of 941

%Straight Lines [id:da5812384143449235] 
\draw    (297,577.38) -- (257.03,667.48) ;
%Straight Lines [id:da33001905163928136] 
\draw    (338.03,576.48) -- (257.03,667.48) ;
%Straight Lines [id:da6616157108921689] 
\draw    (338.03,576.48) -- (298.07,666.58) ;
%Straight Lines [id:da6341349465993615] 
\draw    (377,578.38) -- (337.03,668.48) ;
%Straight Lines [id:da6355113527528056] 
\draw    (377,578.38) -- (378.03,667.48) ;
%Straight Lines [id:da3085098900542519] 
\draw    (377,578.38) -- (257.03,667.48) ;
%Straight Lines [id:da9351268668758784] 
\draw    (418,577.38) -- (257.03,667.48) ;
%Straight Lines [id:da885716382486394] 
\draw    (298.07,666.58) -- (297,577.38) ;
%Straight Lines [id:da06822230292643428] 
\draw    (418,577.38) -- (337.03,668.48) ;
%Straight Lines [id:da5306513735154371] 
\draw    (338.03,576.48) -- (378.03,667.48) ;
%Straight Lines [id:da4684251717939767] 
\draw    (418.03,664.57) -- (338.03,576.48) ;
%Straight Lines [id:da017225414133775874] 
\draw    (418,577.38) -- (418.03,664.57) ;

% Text Node
\draw (242,672) node [anchor=north west][inner sep=0.75pt]  [font=\footnotesize]  {$v_{1}$};
% Text Node
\draw (286,672) node [anchor=north west][inner sep=0.75pt]  [font=\footnotesize]  {$v_{2}$};
% Text Node
\draw (327,672) node [anchor=north west][inner sep=0.75pt]  [font=\footnotesize]  {$v_{3}$};
% Text Node
\draw (370,672) node [anchor=north west][inner sep=0.75pt]  [font=\footnotesize]  {$v_{4}$};
% Text Node
\draw (409,672) node [anchor=north west][inner sep=0.75pt]  [font=\footnotesize]  {$v_{5}$};
% Text Node
\draw (287,556.47) node [anchor=north west][inner sep=0.75pt]  [font=\footnotesize]  {$h_{1}$};
% Text Node
\draw (329,555.47) node [anchor=north west][inner sep=0.75pt]  [font=\footnotesize]  {$h_{2}$};
% Text Node
\draw (367,557.47) node [anchor=north west][inner sep=0.75pt]  [font=\footnotesize]  {$h_{3}$};
% Text Node
\draw (409,556.47) node [anchor=north west][inner sep=0.75pt]  [font=\footnotesize]  {$h_{4}$};

\draw (327,690) node [anchor=north west][inner sep=0.75pt]  [font=\small]  {$G(\mcC)$};

\filldraw[black] (297,577.38)  circle (1.5pt) ;
\filldraw[black] (257.03,667.48) circle (1.5pt) ;
\filldraw[black] (338.03,576.48) circle (1.5pt) ;
\filldraw[black] (377,578.38)  circle (1.5pt) ;
\filldraw[black] (298.07,666.58) circle (1.5pt) ;
\filldraw[black] (378.03,667.48) circle (1.5pt) ;\filldraw[black] (337.03,668.48)  circle (1.5pt) ;
\filldraw[black] (418,577.38) circle (1.5pt) ;
\filldraw[black] (418.03,664.57) circle (1.5pt) ;

\end{tikzpicture}
    \caption{Associated bipartite graph of $\mcC$.}
    \label{fig:g(c)}
\end{figure}
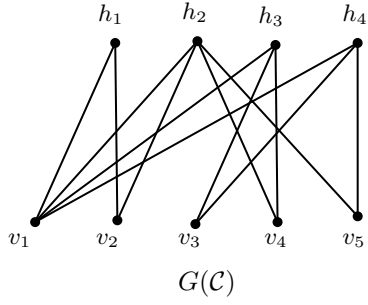
\end{example}

\vspace{2mm}

Note that $|H_i \cap V_j| \leq 1$. For given pair $H_i$ and $V_j$, we denote their intersecting vertex by $x_{a_ib_j}$. To each cycle $\sigma: h_{i_1},v_{j_1},h_{i_2},v_{j_2},\ldots,h_{i_r},v_{i_r}$ in $G(\mathcal{C})$, where $r\geq 2$, we associate a binomial defined by: 
$$f_{\sigma}=x_{a_{i_1},b_{j_1}}\cdots x_{a_{i_r},b_{j_r}} - x_{a_{i_2},b_{j_1}}\cdots x_{a_{i_1},b_{j_r}}.$$ 

Let $J_{\mcC}$ be an ideal generated by all the binomials $f_{\sigma}$, where $\sigma$ is a cycle in $G(\mathcal{C})$.

\begin{example}
    Let $\mcC$ be the corner-interval minors as shown in Figure \ref{fig.maxC}. Then the ideal 
    $$J_\mcC= I(\mcC)+(x_{12}x_{23}x_{34}-x_{14}x_{22}x_{33}).$$
\end{example}

In \cite{Q2012}, it is shown that $J_{\mcC}$ is the toric ideal of the edge ring of $G(\mathcal{C})$; we recall it below. 

\vspace{2mm}

 Let $\mathcal{C}$ be an arbitrary set of corner-interval minors and let $S=K[x_{ij} \mid x_{ij}\in V(\mcC)]$. Let $K[G(\mathcal{C})]=K[h_iv_j \mid \{h_i,v_j\}\in E(G(\mcC))]$ be the subalgebra of polynomial algbera $T=K[h_1,\ldots,h_q,v_1,\ldots,v_p]$. Let $\phi:S \rightarrow K[G(\mathcal{C})]$ be the surjective $K$-algebra homomorphism defined by $\phi(x_{ij})=h_iv_j$, where $x_{ij}=H_i\cap V_j$. We denote by $J_\mcC$ the toric ideal of $K[G(\mathcal{C})]$. Thus, $J_\mcC$ is a prime ideal. From \cite{OH1999} it is known that the kernel $J_\mcC$ of $\phi$ is generated by the binomials $f_\sigma$, where $\sigma$ is a cycle of $G(\mcC)$. 

\vspace{2mm}

The following remarks and Proposition \ref{prop.1} describe some key properties of $J_\mcC$ and $G(\mcC)$.

\begin{remark}
    Let $\mathcal{C}$ be an arbitrary set of corner-interval minors. Then one has $I(\mcC) \subset J_{\mcC}$.
\end{remark}

\begin{remark} \label{remark.1}
    Let $\mathcal{C}$ be an arbitrary set of corner-interval minors. Suppose $V_1=\{x_{11},x_{21},\ldots,x_{m1}\}$ be a maximal vertical interval of $\mcC$. Let $\sigma: h_{i_1},v_{j_1},h_{i_2},v_{j_2},\ldots,h_{i_r},v_{i_r}$ be a cycle in $G(\mathcal{C})$. If $r>2$ and $v_{j_k}=v_1$ for some $k$, then $\sigma$ has a chord. The assertion follows from the fact that every maximal horizontal interval intersects with the vertical interval $V_1$. In particular, the edge $\{h_{i_{k-1}},v_{j_k}\}$ or $\{h_{i_{k+2}},v_{j_k}\}$ forms a chord in the cycle $\sigma$. 
\end{remark}

\begin{remark} \label{remark.corner}
    Let $\mathcal{C}$ be an arbitrary collection of corner minors. Denote by $H_1$ and $V_1$ the maximal horizontal and vertical intervals that contain the vertex $x_{11}$, respectively. Let $\sigma: h_{i_1},v_{j_1},h_{i_2},v_{j_2},\ldots,h_{i_r},v_{i_r}$ be a cycle in $G(\mathcal{C})$, where $r\geq 2$. If $r>2$ and $h_1\in V(\sigma)$ or $v_1\in V(\sigma)$, then $\sigma$ has a chord. In particular, if $\sigma$ is chordless and $h_1\in V(\sigma)$ or $v_1\in V(\sigma)$, then $r=2$.
\end{remark}

\begin{proposition} \label{prop.1}
    Let $\mathcal{C}$ be an arbitrary set of corner-interval minors. Then, for every vertical interval $\{x_{ij},x_{lj}\}$ in $\mcC$, with $j\neq 1$, one has $x_{i1}x_{lj}-x_{l1}x_{ij} \in (I(\mathcal{C}):x^{\infty})$, where $x=\prod_{x_{ij}\in V(\mathcal{C})} x_{ij}$. 
\end{proposition}
\begin{proof}
    Let $\{x_{ij},x_{lj}\}$ be a vertical interval in $\mcC$. If $\{x_{ij},x_{lj}\} \in E(\mcC)$, then it follows that $x_{i1}x_{lj}-x_{l1}x_{ij} \in I(\mcC) \subset (I(\mathcal{C}):x^{\infty})$. Suppose $\{x_{ij},x_{lj}\} \notin E(\mcC)$, then by the definition of vertical interval there exists a sequence of vertical edges $x_{ij}=x_{i_1j},x_{i_2j},\ldots,x_{i_kj}=x_{lj}$ with $\{x_{i_rj},x_{i_{r+1}j}\} \in E(\mcC)$ for $r=1,\ldots,k-1$. We prove by induction on length of the sequence of vertical edges. For $k=3$, we have $f=x_{i_11}x_{i_2j}-x_{i_21}x_{i_1j}$ and $g=x_{i_21}x_{i_3j}-x_{i_31}x_{i_2j}$ belongs to $I(\mcC)$. Then one may write $x_{i_2j}(x_{i_11}x_{i_3j}-x_{i_31}x_{i_1j})=x_{i_3j}f+x_{i_1j}g$. This shows that the assertion holds. Now assume that the statement holds for all $k-1$, where $k \geq 4$. From the induction hypothesis, it follows that $f'=x_{i_11}x_{i_{k-1}j}-x_{i_{k-1}1}x_{i_1j} \in (I(\mathcal{C}):x^{\infty})$. It is known that $g'=x_{i_{k-1}1}x_{i_kj}-x_{i_k1}x_{i_{k-1}j}$ belongs to $I(\mcC)$. Then one may write $x_{i_{k-1}j}(x_{i_11}x_{i_kj}-x_{i_k1}x_{i_1j})=x_{i_kj}f'+x_{i_1j}g'$. This completes the proof.
\end{proof}

In the following lemma, we study the minimal prime ideals of $I(\mcC)$ that do not contain any variables.

\begin{lemma} \label{lemma.primenovar}
    Let $\mathcal{C}$ be an arbitrary collection of corner-interval minors. Then the ideal $J_{\mcC}$ is a minimal prime of $I(\mathcal{C})$.
\end{lemma}
\begin{proof}
   Let $x=\prod x_{ij}$. We claim that $I(\mathcal{C}):x^{\infty}=J_{\mcC}$. Since $I(\mathcal{C}) \subset J_{\mcC}$ and $J_{\mcC}$ is a prime ideal, it follows that $I(\mathcal{C}):x^{\infty} \subseteq J_{\mcC}$. To show the other side containment, let $f_\sigma \in J_{\mcC}$, where $\sigma$ an arbitrary cycle in $G(\mcC)$. We will claim that $f_\sigma \in I(\mathcal{C}):x^{\infty}$. 
   
    Suppose $\sigma$ is a cycle of length $4$ in $G(\mcC)$, and let $f_\sigma=x_{ij}x_{lk}-x_{ik}x_{lj}$ be the corresponding binomial. Then $\{x_{ij},x_{lj}\}$ and $\{x_{ik},x_{lk}\}$ are the two vertical intervals in $\mcC$ corresponding to the cycle $\sigma$. Suppose either $j=1$ or $k=1$. Assume $j=1$. Then, since $\{x_{ik},x_{lk}\}$ is a vertical interval of $\mcC$ with $k\neq 1$, from Proposition \ref{prop.1} it follows that $f_{\sigma} \in (I(\mathcal{C}):x^{\infty})$, which implies the assertion. If $j\neq 1$ and $k\neq 1$, then from Proposition \ref{prop.1} one has the elements $f=x_{i1}x_{lj}-x_{l1}x_{ij}$ and $g=x_{i1}x_{lk}-x_{l1}x_{ik}$ belongs to $(I(\mathcal{C}):x^{\infty})$. Since $x_{l1}f_{\sigma}=x_{lj}g-x_{lk}f$, it follows that $f_{\sigma} \in I({\mcC}):x^{\infty}$. 
   
    Suppose $\sigma: h_{i_1},v_{j_1},h_{i_2},v_{j_2},\ldots,h_{i_r},v_{i_r}$ is a cycle of length $2r$ in $G(\mcC)$ with $r\geq 3$, and $f_{\sigma}=x_{a_{i_1},b_{j_1}}\cdots x_{a_{i_r},b_{j_r}} - x_{a_{i_2},b_{j_1}}\cdots x_{a_{i_{r+1}},b_{j_r}}$, where $i_{r+1}=i_1$, be the corresponding binomial.  Since $\{x_{a_{i_k}b_{j_k}},x_{a_{i_{k+1}}b_{j_k}}\}$ is a vertical interval in $\mcC$ for each $k$ (this vertical interval associated with the edges $\{v_{i_k},h_{i_k}\}$ and $\{v_{i_{k+1}},h_{i_k}\}$ of $\sigma$), Proposition \ref{prop.1} gives $\mu_k=x_{a_{i_k}1}x_{a_{i_{k+1}}b_{j_k}}-x_{a_{i_{k+1}}1}x_{a_{i_k}b_{j_k}} \in (I(\mathcal{C}):x^{\infty})$, whenever $b_{j_k}\ne1$. If $b_{j_k}=1$, then $\mu_k=0$. Then $\mu_k \in (I(\mathcal{C}):x^{\infty})$ for all $k=1,\ldots,r$. 
    
    First, we claim that the element $\lambda_{\ell}= x_{a_{i_{1}}1}x_{a_{i_2}b_{j_1}}\cdots x_{a_{i_\ell}b_{j_{\ell-1}}} - x_{a_{i_{\ell}}1}x_{a_{i_1}b_{j_1}}\cdots x_{a_{i_{\ell-1}}b_{j_{\ell-1}}} \in (I(\mathcal{C}):x^{\infty})$, for $\ell = 2,\ldots,r$. The claim is proved by induction on $\ell$. For $\ell =2$, this is the case by $\mu_1$. Suppose now that $\ell >2$. By applying induction hypothesis, we have that $\lambda_{\ell-1} =x_{a_{i_{1}}1}x_{a_{i_2}b_{j_1}}\cdots x_{a_{i_{\ell-1}}b_{j_{\ell-2}}} - x_{a_{i_{\ell-1}}1}x_{a_{i_1}b_{j_1}}\cdots x_{a_{i_{\ell-2}}b_{j_{\ell-2}}}$ belongs to $(I(\mathcal{C}):x^{\infty})$. Since $\lambda_\ell= x_{a_{i_{\ell}}b_{j_{\ell-1}}}\lambda_{\ell-1}+x_{a_{i_1}b_{j_1}}\cdots x_{a_{i_{\ell-2}}b_{j_{\ell-2}}}\mu_{\ell-1}$, it follows that $\lambda_\ell \in (I(\mathcal{C}):x^{\infty})$. Then, one may write $x_{a_{i_{1}}}f_\sigma=-x_{a_{i_{1}}b_{j_{r}}}\lambda_r - x_{a_{i_1}b_{j_1}}\cdots x_{a_{i_{r-1}}b_{j_{r-1}}}\mu_r$, thus it follows that $f_\sigma \in I(\mathcal{C}):x^{\infty}$. Hence, we prove the claim.

   Now, if $P$ is a minimal prime ideal of $I(\mcC)$ containing no variable, such a prime exists since $J_\mcC$ is a prime ideal, then $ J_{\mcC}=I(\mathcal{C}):x^{\infty} \subset P:x^{\infty}=P$, hence $J_{\mcC}$ is a minimal prime of $I(\mathcal{C})$.
\end{proof}

\begin{corollary}
    Let $\mathcal{C}$ be an arbitrary corner-interval minors. Then $I(\mcC)$ is prime if and only if $I(\mcC)=J_\mcC$.
\end{corollary}

\vspace{2mm}

Next, to study the minimal prime ideals of $I(\mcC)$ containing variables, we first present the following definitions.

\vspace{2mm}

\begin{definition}
    Let $\mcC=\{\delta_1,\delta_2,\ldots,\delta_r\}$ be a set of corner-interval minors. A subset $W\subset V(\mcC)$ is called admissible if for each index $i$ either $W\cap V(\delta_i)=\emptyset$ or $W\cap V(\delta_i)\neq \emptyset$ contains an edge of $\delta_i$. 
\end{definition}

\begin{definition} \label{def.pw}
    Let  $W\subset V(\mcC)$ be an admissible set. We define an ideal $P_W(\mcC)$ as follows:
    $$P_W(\mcC)=(x_{ij}\mid x_{ij}\in W) + J_{\mcC'} \subset S,$$
    where $\mcC'=\{\delta \in \mcC \mid V(\delta)\cap W= \emptyset\}$.
\end{definition}

Notice that for an admissible set $W$ of $\mcC$, the associated ideal satisfies $P_W(\mcC)=W+P_\emptyset(\mcC')$, where $\mcC'=\{\delta \in \mcC \mid V(\delta)\cap W= \emptyset\}$.

\begin{example}
    Let $\mcC$ be the corner-interval minors as shown in Figure \ref{fig.maxC}. Then the admissible sets of $\mcC$ are the following:
    $$\emptyset,\{x_{12},x_{22}\},\{x_{32},x_{42}\},\{x_{23},x_{33}\},\{x_{14},x_{34}\},\{x_{41},x_{42}\},\{x_{11},x_{12},x_{14}\},\ldots,V(\mcC).$$ 
    For instance, if $W=\{x_{12},x_{22}\}$ then 
    \begin{equation*}
        \begin{split}
            P_{W}(\mcC)&=(x_{12},x_{22})+J_{\mcC'}, \text{ where } \mcC'=\{\delta \in \mcC \mid V(\delta)\cap W= \emptyset\},\\
            &=(x_{12},x_{22},x_{23}x_{31}-x_{21}x_{33},x_{14}x_{31}-x_{11}x_{34},x_{32}x_{41}-x_{31}x_{42}).
        \end{split}
    \end{equation*}
\end{example}

\begin{remark} \label{remak.prime}
    Let $\mathcal{C}$ be an arbitrary corner-interval minors. Then the ideal $P_W(\mcC)$ is a prime ideal, for every admissible set $W$ of $\mcC$. Since the ideals $J_{\mcC'}$ and $W$ are prime and generated in disjoint set of variables. 
\end{remark}

The subsequent results highlights the significance of admissible sets in the study of the minimal prime ideals of $I(\mcC)$.

\begin{lemma} \label{lemma.admissible}
    Let $\mathcal{C}$ be an arbitrary set of corner-interval minors. Let $P$ be a prime ideal containing $I(\mcC)$, and let $W=\{x_{ij} \mid x_{ij}\in P\}$. Then $W$ is an admissible set.
\end{lemma}
\begin{proof}
    Let $f_{\delta}=x_{ij}x_{kl}-x_{kj}x_{il}$ be an element of $I(\mcC)$. Suppose that $W \cap V(\delta) \neq \emptyset$. Say $x_{ij} \in W$. Then $x_{kj}x_{il} \in P$. Since $P$ is a prime, it follows that either $x_{kj} \in P$ or $x_{il} \in P$. This implies that $W$ contains the edge either $\{x_{ij},x_{kj}\}$ or $\{x_{ij},x_{il}\}$ of $\delta$.
\end{proof}

\begin{theorem} \label{thm.prime}
    Let $\mathcal{C}$ be an arbitrary set of corner-interval minors and $P$ be a minimal prime ideal of $I(\mcC)$. Then there exists an admissible set $W\subset V(\mcC)$ such that $P=P_W(\mcC)$.
\end{theorem}
\begin{proof}
    Let $P$ be a minimal prime of $I(\mcC)$. From Lemma \ref{lemma.admissible} it follows that there exists a set $W$ such that $(W,I(\mcC)) \subset P$. From the definition of $W$, one has $(W,I(\mcC))=(W,I(\mcC'))$, where $\mcC'=\{\delta \in \mcC \mid V(\delta)\cap W = \emptyset\}$. From Remark \ref{remak.prime} it follows that the ideal $(W,J_{\mcC'})=P_W(\mcC)$ is a prime ideal containing $(W,I(\mcC'))$. Thus, it suffices to prove that $(W,J_{\mcC'}) \subset P$. Since $\mcC'$ is a corner-interval minors and $P$ modulo $W$ is a minimal prime ideal containing $I(\mcC')$, it follows from the virtue of Lemma \ref{lemma.primenovar} that $J_{\mcC'} \subset P \mod W$, as desired.
\end{proof}

\begin{lemma} \label{lemma.primecontainment}
    Let $\mathcal{C}$ be an arbitrary corner-interval minors. Let $W$ and $V$ be two admissible sets of $\mcC$, and let $P_W(\mcC)=(W,J_{\mcC'})$ and $P_V(\mcC)=(V,J_{\mcC''})$ be the corresponding prime ideals, where $\mcC'=\{\delta \in \mcC \mid V(\delta)\cap W= \emptyset\}$ and $\mcC''=\{\delta \in \mcC \mid V(\delta)\cap V= \emptyset\}$ (see Definition \ref{def.pw}). Then the following conditions are equivalent:
    \begin{enumerate}
        \item[(a)] $P_W(\mcC) \subset P_V(\mcC)$.
        \item[(b)] $W \subset V$ and for all elements $f_{\sigma} \in J_{\mcC'}\setminus J_{\mcC''}$ one has that the vertices $x_{a_{i_k},b_{j_k}}$ and $x_{a_{i_{l+1}},b_{j_l}}$ belongs to $V$ for some $k,l$, where $\sigma: h_{i_1},v_{j_1},h_{i_2},v_{j_2},\ldots,h_{i_r},v_{i_r}$ a cycle in $G(\mathcal{C'})$, $r\geq 2$ and $i_{r+1}=i_{1}$.
    \end{enumerate}
\end{lemma}
\begin{proof}
    $(a)\implies (b)$. Let $x_{ij} \in W$. Since $x_{ij} \in P_W(\mcC) \subset P_V(\mcC)$, it follows that $x_{ij} \in V$. Therefore, $W \subset V$.  Let $f_{\sigma} \in J_{\mcC'}\setminus J_{\mcC''}$, where $\sigma$ a cycle in $G(\mcC')$, be an element of $J_{\mcC'}$. The inclusion $(a)$ implies that for all $f_{\sigma} \in J_{\mcC'}$ we have $f_{\sigma} \in (V,J_{\mcC''})$. Suppose $x_{a_{i_k},b_{j_k}}$ and $x_{a_{i_{k+1}},b_{j_k}}$ does not contain a variable of $V$ for all $k=1,\ldots,r$, where $i_{r+1}=i_{1}$.  Then $f_{\sigma} \in J_{\mcC''}$, a contradiction.  If exactly one of them contains a variable of $V$, then a nonzero monomial belongs to $J_{\mcC''}$, again a contradiction, since $J_{\mcC''}$ is a toric ideal.  This implies that there exists $k$ and $\ell$ such that $x_{a_{i_k},b_{j_k}}$ and $x_{a_{i_{\ell+1}},b_{j_\ell}}$ belongs to $V$, as required.

    $(b)\implies (a)$. Suppose for all $f_{\sigma} \in J_{\mcC'}\setminus J_{\mcC''}$ we have the vertices $x_{a_{i_k},b_{j_k}}$ and $x_{a_{i_{l+1}},b_{j_l}}$ belongs to $V$ for some $k,l$. Then it follows that $f_{\sigma} \in P_V(\mcC)$. Thus, combined with the inclusion $W \subset V$, it follows that $P_W(\mcC) \subset P_V(\mcC)$.
\end{proof}

We are now in a position to present the main theorem of this section.

\begin{theorem} \label{thm.minimal}
    Let $\mathcal{C}$ be an arbitrary corner-interval minors. Let $W$ and $V$ be two admissible sets of $\mcC$, and let $P_W(\mcC)=(W,J_{\mcC'})$ and $P_V(\mcC)=(V,J_{\mcC''})$ be the corresponding prime ideals. Then $P_V(\mcC)$ is a minimal prime for $I(\mcC)$ if and only if for all admissible subsets $W \subsetneq V$, there exists $f_{\sigma} \in J_{\mcC'}\setminus J_{\mcC''}$ such that the vertices $x_{a_{i_k},b_{j_k}}$ or $x_{a_{i_{l+1}},b_{j_l}}$ does not belongs to $V$, for all $1 \leq k,l \leq r$, where $\sigma: h_{i_1},v_{j_1},h_{i_2},v_{j_2},\ldots,h_{i_r},v_{i_r}$ a cycle in $G(\mathcal{C'})$, $r\geq 2$ and $i_{r+1}=i_1$.
\end{theorem}
\begin{proof}
    It follows from Theorem \ref{thm.prime} and Lemma \ref{lemma.primecontainment}. 
\end{proof}

\section{Corner minors} \label{sec.cor}

In this section, we focus on binomial ideals generated by corner minors, with the objective of characterizing their radicality. As a first step, we examine the structure of their minimal prime ideals below.

\begin{lemma} \label{lemma.w11}
    Let $\mcC$ be an arbitrary corner minors of $M$. Let $P_W(\mcC)$ be a minimal prime of $I(\mcC)$. If the admissible set $W \neq \emptyset$, then $x_{11} \in W$. In particular, if $W \neq \emptyset$ then the minimal prime $P_W(\mcC)$ is generated by variables.
\end{lemma}
 
\begin{proof}
    Let $W \neq \emptyset$ and $P_W(\mcC)$ a minimal prime of $I(\mcC)$. Suppose $x_{11} \notin W$. Then we have $x_{11} \notin P_W(\mcC)$. We show that $J_{\mcC} \subseteq P_W(\mcC)$. Let $\sigma: h_{i_1},v_{j_1},h_{i_2},v_{j_2},\ldots,h_{i_r},v_{i_r}$ be a cycle in $G(\mcC)$ where $r\geq 2$ and $i_{r+1}=i_i$ and $f_{\sigma}=x_{a_{i_1},b_{j_1}}\cdots x_{a_{i_r},b_{j_r}} - x_{a_{i_2},b_{j_1}}\cdots x_{a_{i_{r+1}},b_{j_r}}$, where $i_{r+1}=i_1$, be the corresponding binomial.

    Since $\mcC$ consists of corner minors, for every $x_{ij}\in V(\mcC)$, we have $x_{11}x_{ij}-x_{i1}x_{1j}\in I(\mcC)$. Consequently, $x_{11}^r (x_{a_{i_1}b_{j_1}}\cdots x_{a_{i_r}b_{j_r}}) - x_{a_{i_1}1}x_{1b_{1_k}}\cdots x_{a_{i_r}1}x_{1b_{j_r}} \in I(\mcC)$ and $x_{11}^r (x_{a_{i_{2}}b_{j_1}}\cdots x_{a_{i_{r+1}}b_{j_r}}) -x_{a_{i_{2}}1}x_{1b_{j_r}}\cdots x_{a_{i_{r+1}}1}x_{1b_{j_r}} \in I(\mcC)$. Subtracting these two expressions gives $x_{11}^r f_\sigma\in I(\mcC) \subseteq P_W(\mcC)$. 

    Since $P_W(\mcC)$ is a prime and $x_{11}\notin P_W(\mcC)$, it follows that  $f_\sigma\in P_W(\mcC)$. Since $J_{\mcC}$ is generated by the cycle binomials, it follows that $J_{\mcC}\subseteq P_W(\mcC)$. Since $J_{\mcC}$ is a prime ideal containing $I(\mcC)$ and from the minimality of $P_W(\mcC)$, it follows that $P_W(\mcC)=J_{\mcC}$. This is a contradiction, since $W\neq\emptyset$ and $J_{\mcC}$ contains no variables. Therefore $x_{11}\in W$.

    The second part of the assertion follows from Lemma \ref{lemma.admissible}, since for every element $x_{11}x_{ij}-x_{1j}x_{i1}$ in $I(\mcC)$, either $x_{11},x_{i1} \in W$ or $x_{11},x_{1j} \in W$. Therefore, there are no cycles in $G(\mcC')$, where $\mcC'=\{\delta \in \mcC \mid V(\delta)\cap W= \emptyset\} = \emptyset$, which implies that $P_W(\mcC)$ is generated by variables. This completes the proof.
\end{proof}

\begin{notation}
    Let $\mcC$ be an arbitrary collection of corner-minors of the matrix $M$. We denote by $\mathcal{W}$ the set of all admissible subsets $W\subset V(\mcC)$ such that the ideal $P_W(\mcC)$ is a minimal prime ideal of $I(\mcC)$. 
\end{notation}

Let $<$ be the reverse lexicographic order on $S=K[x_{ij}\mid x_{ij}\in V(\mcC)]$ induced by $x_{ij}<x_{kl}$ if $i<k$, or $i=k$ and $j<l$.

\begin{corollary} \label{cor.intersW}
    Let $\mcC$ be an arbitrary collection of corner minors of $M$.
    Then
    \begin{equation*}
        \bigcap_{W\in \mathcal{W} \setminus \{\emptyset\}} P_W(\mcC) = ((x_{11})+(\ini(f_\delta) \mid \delta \in \mcC)):J_\mcC.
    \end{equation*}
\end{corollary}

\begin{proof}
Set $Q=I(\mcC)+(x_{11})$. Since $\mcC$ consists of corner minors,  $Q=(x_{11})+ (x_{i1}x_{1j}: [1,i\mid1,j]\in\mcC)$. Thus $Q$ is a squarefree monomial ideal and hence radical. Consequently, $Q=\cap_{P\in\operatorname{Min}(Q)}P$.

By Theorem \ref{thm.prime} and Lemma \ref{lemma.w11}, the minimal primes of $I(\mcC)$ are $J_{\mcC}$ and those minimal primes $P$ of $Q$ for which $J_{\mcC}\nsubseteq P$. Indeed, if $J_{\mcC}\subseteq P$, then $P$ is not minimal over $I(\mcC)$, since $x_{11}\in P\setminus J_{\mcC}$. Conversely, if $J_{\mcC}\nsubseteq P$, then $P$ cannot properly contain a minimal prime of $I(\mcC)$.

Therefore,
\begin{equation*}
    \begin{aligned}
\bigcap_{W\in\mathcal W\setminus\{\emptyset\}} P_W(\mcC)
&= \bigcap_{\substack{ P\in\operatorname{Min}(Q)\\ J_{\mcC}\nsubseteq P}}P\\
&=\bigcap_{P\in\operatorname{Min}(Q)}(P:J_{\mcC})\\
&=Q:J_{\mcC}.
\end{aligned}
\end{equation*}

\end{proof}

\begin{notation}
    Let $\mcC$ be an arbitrary set of corner minors of $M$. For a given cycle 
    $\sigma:= h_{i_1},v_{j_1},h_{i_2},v_{j_2},\ldots,h_{i_r},v_{i_r}$ in $G(\mathcal{C})$, we denoted by $V(\sigma)$ the vertex set of $\sigma$.
\end{notation}

\begin{lemma} \label{lemma.x11f}
    Let $\mcC$ be an arbitrary set of corner minors of $M$, and let $\sigma$ be a cycle in the associated graph $G(\mcC)$. Denote by $H_1$ and $V_1$ the maximal horizontal and vertical intervals of $\mcC$ that contain the vertex $x_{11}$, respectively. Then $x_{11}f_\sigma \in I(C)$ if and only if $V(\sigma) \cap v_1 \neq \emptyset$ or $V(\sigma) \cap h_1 \neq \emptyset$.
\end{lemma}
\begin{proof}
Let $\sigma: h_{i_1},v_{j_1},h_{i_2},\ldots, h_{i_r},v_{j_r}$ be a cycle in $G(\mcC)$, where $r\geq 2$.  Suppose $V(\sigma) \cap v_1 \neq \emptyset$ or $V(\sigma) \cap h_1 \neq \emptyset$. Then from Remark \ref{remark.corner} it follows that $r=2$. Without loss of generality, assume that $f_\sigma \notin I(\mcC)$ and  $V(\sigma) \cap v_1 \neq \emptyset$ and $v_{j_1}=v_1$. Set $\sigma: h_{i_1},v_{1},h_{i_2},v_{j_2}$ be a cycle in $G(\mcC)$. Since $x_{a_{i_1}b_{j_2}}$ and $x_{a_{i_2}b_{j_2}}$ belongs to $V(\mcC)$, it follows that $\sigma_1: h_{1},v_{1},h_{i_2},v_{j_2}$ and $\sigma_2: h_{1},v_{1},h_{i_1},v_{j_2}$ are cycles in $G(\mcC)$ (in particular, $f_{\sigma_1}$ and $f_{\sigma_2}$ belongs to $I(\mcC)$). Then $x_{a_1,b_1}f_\sigma=x_{a_{i_1},b_1}f_{\sigma_1}-x_{a_{i_2},b_1}f_{\sigma_2}$, where $a_1=1$ and $b_1 =1$. Therefore, we have $x_{11}f_\sigma \in I(C)$.

Conversely, suppose that $V(\sigma) \cap v_1 = \emptyset$ and $V(\sigma) \cap h_1 = \emptyset$, and $x_{11}f_{\sigma} \in I(\mcC)$. Remark \ref{rem.bwalk} implies that there exits a $\mathcal{B}_\mcC$-walk from $\alpha=x_{11}x_{a_{i_1},b_{j_1}}\cdots x_{a_{i_r},b_{j_r}}$ to $\beta = x_{11}x_{a_{i_1},b_{j_r}}x_{a_{i_2},b_{j_1}} \cdots x_{a_{i_{r}},b_{j_{r-1}}}$. Since $x_{11}$ is the only variable in $\alpha$ which intersects with $H_1$ and $V_1$, it follows that the only possible $\mathcal{B}_\mcC$-walk from $\alpha$ is the following.
\begin{equation*}
    \alpha \rightarrow \alpha_k = x_{a_{i_k},1}x_{1,b_{j_k}}x_{a_{i_1},b_{j_1}}\cdots x_{a_{i_{k-1}},b_{j_{k-1}}}x_{a_{i_{k+1}},b_{j_{k+1}}}\cdots x_{a_{i_r},b_{j_r}}
\end{equation*}
for some $1\leq k \leq r$. Since $x_{a_{i_k},1}$ is the only variable in $\alpha_k$ which intersects with $V_1$ and $x_{1,b_{j_k}}$ is the only variable in $\alpha_k$ which intersects with $H_1$, it follows that the only possible  $\mathcal{B}_\mcC$-walk from $\alpha_k$ is the walk $\alpha_k \rightarrow \alpha$. Therefore, every $\mathcal{B}_\mcC$-walk from $\alpha$ returns to $\alpha$ via $\alpha_k$, forming a loops of the form
\begin{equation*}
    \alpha \rightarrow \alpha_k \rightarrow \alpha
\end{equation*}
for all $1\leq k \leq r$. Note that $\alpha_k \neq \beta$ for all $k$. This contradicts the existence of a $\mathcal{B}_\mcC$-walk from $\alpha$ to $\beta$. Hence, the theorem is proved.
\end{proof}

Since the ideal $P_W(\mcC)$ is generated by variables when $W\neq \emptyset$, we can explicitly compute their intersection in the following theorem. 

\begin{theorem} \label{lemma.rad}
    Let $\mcC$ be an arbitrary collection of corner minors of $M$. Let $\mathcal{W}$ be the set of all admissible sets $W$ such that $P_W(\mcC)$ is a minimal prime of $I(\mcC)$.  Then one has
    \begin{equation} \label{eq.pw}
        \bigcap_{W \in \mathcal{W}} P_W(\mcC)=I(\mcC)+(x_{11}f_\sigma \mid \sigma \text{ a cycle in } G(\mcC) \text{ such that } V(\sigma) \cap v_1 = \emptyset \text{ and } V(\sigma) \cap h_1 = \emptyset),
    \end{equation} where $H_1$ and $V_1$ are maximal horizontal and maximal vertical intervals of $\mcC$ containing the vertex $x_{11}$, respectively.
\end{theorem}

\begin{proof}
Set $Q=I(\mcC)+(x_{11})$. By Corollary~\ref{cor.intersW}, we have $\bigcap_{W\in\mathcal W\setminus\{\emptyset\}} P_W(\mcC)=Q:J_\mcC$. Moreover, $P_\emptyset(\mcC)=J_\mcC$. Consequently, $\bigcap_{W\in\mathcal W}P_W(\mcC) =J_\mcC\cap(Q:J_\mcC)$. Since $Q$ is radical it follows that  $J_\mcC\cap(Q:J_\mcC)=J_\mcC\cap Q$.

We show that $J_\mcC\cap Q=I(\mcC)+x_{11}J_\mcC$. Clearly, $I(\mcC)+x_{11}J_\mcC \subseteq J_\mcC\cap Q$. For the reverse inclusion, let $g\in J_\mcC\cap Q$. Since $Q=I(\mcC)+(x_{11})$, we may write $g=f+x_{11}h$, $f\in I$, $h\in S$. As $I(\mcC)\subseteq J$, we have $x_{11}h=g-f\in J_\mcC$. Since $J_\mcC$ is prime and $x_{11}\notin J_\mcC$, it follows that $h\in J_\mcC$. Hence $g\in I(\mcC)+x_{11}J_\mcC$, as required.

Therefore, $\bigcap_{W\in\mathcal W}P_W(\mcC) =I+(x_{11}f_\sigma\mid \sigma\text{ is a cycle in }G(\mcC))$. It remains to show that the generators $x_{11}f_\sigma$ corresponding to cycles meeting $h_1$ or $v_1$ belong to $I(\mcC)$.

We prove this by induction on the length of $\sigma$. Suppose first that $\sigma$ is a $4$-cycle. If $\sigma$ contains both $h_1$ and $v_1$, then $f_\sigma\in I(\mcC)$. If $\sigma$ contains $h_1$ but not $v_1$, then $f_\sigma=x_{1j}x_{ik}-x_{1k}x_{ij}$, and $x_{11}f_\sigma =x_{1j}(x_{11}x_{ik}-x_{i1}x_{1k})-x_{1k}(x_{11}x_{ij}-x_{i1}x_{1j}) \in I(\mcC)$. Similarly, if $\sigma$ contains $v_1$ but not $h_1$, then $x_{11}f_\sigma \in I(\mcC)$. 

Now suppose that $\sigma$ has length $2r\geq6$ and meets $h_1$ or $v_1$. From Remarks \ref{remark.1} and \ref{remark.corner} it follows that every maximal vertical interval intersects $H_1$, and every maximal horizontal interval intersects $V_1$, the cycle $\sigma$ has a chord incident to $h_1$ or $v_1$.

This chord divides $\sigma$ into two shorter even cycles $\sigma_1$ and $\sigma_2$, each containing the chosen vertex $h_1$ or $v_1$. By the standard chord decomposition of cycle binomials, $f_\sigma=a f_{\sigma_1} + bf_{\sigma_2}$ for some monomials $a,b\in S$.
By the induction hypothesis, $x_{11}f_{\sigma_1},  x_{11}f_{\sigma_2}\in I(\mcC)$. Consequently, $x_{11}f_\sigma\in I(\mcC)$.

Thus every cycle binomial $f_\sigma$ whose cycle meets $h_1$ or $v_1$ satisfies $x_{11}f_\sigma\in I$, as desired.
\end{proof}

\begin{example}
    Let $I(\mcC)$ be the corner minor ideal generated by $$x_{12}x_{21}-x_{11}x_{22},x_{14}x_{21}-x_{11}x_{24},x_{31}x_{12}-x_{11}x_{32},x_{31}x_{13}-x_{11}x_{33},x_{41}x_{13}-x_{11}x_{43},x_{41}x_{14}-x_{11}x_{44}.$$ Then the set \begin{equation*}
    \begin{split}
        \mathcal{W}=&\{\emptyset,\{x_{11},x_{12},x_{13},x_{14}\},\{x_{11},x_{21},x_{31},x_{41}\},\{x_{11},x_{12},x_{14},x_{31},x_{41}\},\\&\{x_{11},x_{12},x_{13},x_{21},x_{41}\},\{x_{11},x_{13},x_{14},x_{21},x_{41}\}\}.
    \end{split}
    \end{equation*}
    Moreover, the intersection of minimal primes $\bigcap_{W \in \mathcal{W}} P_W(\mcC)=I(\mcC)+x_{11}(x_{22}x_{33}x_{44}-x_{32}x_{43}x_{24})$.
\end{example}

As a consequence, we provide a characterization of the radicality and the primary decomposition of the ideal generated by an arbitrary set of $2$-corner minors of $M$, as below.

\begin{corollary} \label{thm.rad}
    Let $\mcC$ be an arbitrary corner minors of $M$. Denote by $H_1$ and $V_1$ the maximal horizontal and vertical intervals that contain the vertex $x_{11}$, respectively. Then $I(\mcC)$ is radical if and only if $G(\mcC)$ does not have a cycle in the induced subgraph on vertex set $V(G(\mcC))\setminus \{v_1,h_1\}$. Moreover, if $G(\mcC)$ does not have a cycle in the induced subgraph on vertex set $V(G(\mcC))\setminus \{v_1,h_1\}$ then $ \bigcap_{W \in \mathcal{W}}P_W(\mcC)$ is the primary decomposition of $I(\mcC)$. 
\end{corollary}

\begin{proposition} \label{prop.decomp}
        Let $\mcC$ be an arbitrary corner minors of $M$. Then one has 
    \begin{equation} \label{eq.thm}
       I(\mcC)= \bigcap_{W \in \mathcal{W}} P_W(\mcC) \cap (I(\mcC)+x_{11}^2).
    \end{equation}
\end{proposition}

\begin{proof}
    First, we claim that $x_{11}^2J_\mcC \subseteq I(\mcC)$.  Let $\sigma: h_{i_1},v_{j_1},h_{i_2},\ldots, h_{i_r},v_{j_r}$ be a cycle in $G(\mcC)\setminus v_1,h_1$, where $r\geq 2$. Observe that for  each $k$ the vertices $x_{a_{i_k},b_{j_k}}$ and $x_{a_{i_k},b_{j_{k-1}}}$, one has $[1,a_{i_k}|1,b_{j_k}]$ and $[1,a_{i_k}|1,b_{j_{k-1}}]$ belongs to $\mcC$, where $j_0=j_r$. By Remark \ref{rem.bwalk} it suffices to show the existence of a $\mathcal{B}_\mcC$-walk from $x_{11}^2x_{a_{i_1},b_{j_1}}\cdots x_{a_{i_r},b_{j_r}}$ to $x_{11}^2x_{a_{i_1},b_{j_r}}x_{a_{i_2},b_{j_1}}\cdots x_{a_{i_{r}},b_{j_{r-1}}}$. We construct a $\mathcal{B_C}$-walk below:
    
    \begin{gather*}
        x_{11}^2x_{a_{i_1},b_{j_1}}\cdots x_{a_{i_r},b_{j_r}} 
        \rightarrow 
        x_{11}(x_{1,b_{j_1}}x_{a_{i_1},1})x_{a_{i_2},b_{j_2}}\cdots x_{a_{i_r},b_{j_r}}  
        \rightarrow \\ 
         (x_{1,b_{j_1}}x_{a_{i_1},1})(x_{1,b_{j_r}}x_{a_{i_r},1})x_{a_{i_2},b_{j_2}}\cdots x_{a_{i_{r-1}},b_{j_{r-1}}} \\
        \downarrow (\text{ since } [1,a_{i_1}|1,b_{j_{r}}] \in \mcC)\\
        x_{11}(x_{a_{i_1},b_{j_r}})(x_{1,b_{j_1}})(x_{a_{i_r},1})x_{a_{i_2},b_{j_2}}\cdots x_{a_{i_{r-1}},b_{j_{r-1}}} \rightarrow\\
          \vdots  \\
         \rightarrow x_{11}(x_{a_{i_1},b_{j_r}})(x_{a_{i_2},b_{j_1}})\cdots(x_{a_{i_{r-1}},b_{j_{r-2}}})(x_{a_{i_r},1})(x_{1,b_{j_{r-1}}})\\
         \downarrow (\text{ since } [1,a_{i_{r}}|1,b_{j_{r-1}}] \in \mcC)\\
         x_{11}^2(x_{a_{i_1},b_{j_r}})(x_{a_{i_2},b_{j_1}})\cdots(x_{a_{i_{r}},b_{j_{r-1}}}).
    \end{gather*} 
    Hence $x_{11}^2f_\sigma \in I(\mcC)$. If $\sigma$ contains $h_1$ or $v_1$, then $x_{11}f_\sigma\in I(\mcC)$ by Lemma~\ref{lemma.x11f}. Since $J_\mcC$ is generated by the cycle binomials, the claim follows. Consequently, $J_\mcC \cap (I(\mcC)+x_{11}^2)=I(\mcC)$. Indeed, if $g=f+x_{11}^2h\in J_\mcC$, where $f\in I(\mcC)$, then $x_{11}^2h\in J_\mcC$. Since $J_\mcC$ is prime and $x_{11}\notin J_\mcC$, we have $h\in J_\mcC$. Hence $x_{11}^2h\in I(\mcC)$.

    By Theorem \ref{thm.prime}, we have $\cap_{W\in\mathcal W}P_W(\mcC) =\sqrt{I(\mcC)}\subseteq J_\mcC$. Consequently, $I(\mcC)\subseteq (\bigcap_{W\in\mathcal W}P_W(\mcC)) \cap(I(\mcC)+(x_{11}^2))\subseteq J_\mcC \cap(I(\mcC)+(x_{11}^2))=I$. This proves the assertion.
\end{proof}

In the following corollary, we discuss the connectedness of contingency tables with respect to an arbitrary set of corner moves.

\begin{corollary}
     Let $\mcC$ be an arbitrary set of corner minors of $M$. Let $U=(a_{ij} \mid x_{ij} \in V(\mcC))$ and $V=(b_{ij} \mid x_{ij} \in V(\mcC))$ be two contingency tables. Then $U$ and $V$ are connected via $\mathcal{B}_\mcC$ if  
     \begin{enumerate}
         \item \label{cor1} the inequalities $$\sum_{x_{ij}\in W}a_{ij} \geq 2 \text{ and }\sum_{x_{ij}\in W}b_{ij} \geq 2$$ hold, for all $W\in \mathcal{W}\setminus \{\emptyset\}$; and 
         \item \label{cor2} the tables $U$ and $V$ have the same row and column sums such that the non zero entries of $U$ and $V$ corresponds to edges of a cycle in $G(\mcC)$. 
     \end{enumerate}
\end{corollary}
\begin{proof}
    By Remark \ref{rem.bwalk} $U$ and $V$ are connected with respect to $\mathcal{B_C}$ moves if and only if the binomial $f=\prod_{x_{ij}\in V(\mcC)}x_{ij}^{a_{ij}}-\prod_{x_{ij}\in V(\mcC)}x_{ij}^{b_{ij}}$ belongs to $I(\mcC)$. By condition \eqref{cor2} the two monomials of $f$ have the same image under the toric homomorphism associated with $G(\mcC)$. Hence $f\in P_{\emptyset}(\mcC)$. By condition \eqref{cor1}, both monomials of $f$ belong to $P_W(\mcC)^2$ for every $W\in\mathcal W\setminus\{\emptyset\}$. Thus $f\in \cap_{W\in\mathcal W\setminus\{\emptyset\}} P_W(\mcC)^2 \subseteq \cap_{W\in\mathcal W\setminus\{\emptyset\}} P_W(\mcC)$.

    Let $G$ be the bipartite graph with vertex set $\{H=\{x_{i1}\}\sqcup V=\{x_{1j}\} \mid [1,i\mid1,j]\in\mcC\}$ and $E(G)=\{\{x_{i1},x_{1j}\}\mid [1,i\mid1,j]\in\mcC\}$. If $G$ has no cycle, then by Lemma~\ref{lemma.x11f} and Theorem~\ref{lemma.rad}, we have $I(\mcC)=\cap_{W\in\mathcal W}P_W(\mcC)$. Consequently, $f\in I(\mcC)$.

    Now suppose that $G$ contains a cycle. Set $Q=I(\mcC)+(x_{11})=(x_{11})+I(G)$. From Lemma \ref{lemma.w11} and the proof of Corollary \ref{cor.intersW}, it follows that $\cap_{W\in\mathcal W\setminus\{\emptyset\}} P_W(\mcC)^2=\bigcap_{D\in\mathcal{V}(G)}(x_{11},D)^2$, where $\mathcal{V}(G)$ is the set of all minimal vertex cover of $G$. Since $G$ is bipartite, its edge ideal is normally torsion-free by \cite[Corollary 10.3.17]{HH2011}. Therefore, $\cap_{D\in\mathcal{V}(G)}(x_{11},D)^2= (x_{11}^2)+x_{11}I(G)+I(G)^2=Q^2$. Then it follows that $f\in Q^2\subseteq I+(z^2)$.  Hence, by Proposition \ref{prop.decomp}, $f\in (\cap_{W\in\mathcal W}P_W(\mcC)) \cap(I(\mcC)+(x_{11}^2))=I(\mcC)$. Thus, in either case, $f\in I(\mcC)$. 
\end{proof}

\section{Betti numbers} \label{sec.bet}

In this section, we compute the Hilbert–Poincar\'e polynomial of the generalized binomial edge ideals of star graphs. We then compute the Castelnuovo–Mumford regularity of the ideal generated by binomials corresponding to all corner minors of $M$. Finally, we compare the graded Betti numbers of powers of the ideal generated by binomials associated with arbitrary corner-interval minors of $M$ with those of its subcorner-interval minors. We begin by introducing the notation that will be used throughout this section.

\begin{notation} \label{nota.i_t1}
    Denote by $M_{r,s}:= (x_{ij})_{i=r,\ldots,m \atop j=s,\ldots,n}$ the submatrix of $M$. 
    \begin{itemize}
        \item We denote by $I_{\mcC}:= \{x_{i1}x_{jk}-x_{j1}x_{ik} \mid 1 \leq i<j \leq m, 2 \leq k \leq n\}$ the minimal generating set of all corner-interval minors of $M$.
        \item The set of all $2$-minors of the submatrix $M_{r,s}$ is denoted by $I_{r,s}=\{x_{ij}x_{kl}-x_{il}x_{kj} \mid  r \leq i<k \leq m,  s \leq j<l \leq n\}$, and the ideal generated by $I_{r,s}$ is denoted by $I_2(M_{r,s})$.
    \end{itemize} 
\end{notation}

\begin{remark} \cite[Theorem 2]{R2013}
    \label{rem.gr}
    Let $\mcC$ be the set of all corner-interval minors of $M$. Let $<$ be the reverse lexicographic order on $S$ induced by $x_{11}<\cdots<x_{1n}<x_{21}<\cdots<x_{mn}$. Then the following set of binomials 
    $$\mathcal{G}= I_{\mcC} \cup \bigcup_{r=1}^{m-1} x_{r1}I_{r,2}$$
    is a Gr\"obner basis of $I(\mcC)$.
\end{remark} 

\begin{theorem} \label{thm.hpp}
    Let $\mcC$ be the set of all corner-interval minors of $M$. Then the Hilbert-Poincar\'e polynomial of $S/I(\mcC)$ is 
    \begin{equation*}
        \hp_{S/ I(\mcC)}(z)=\sum_{t=1}^{m-1} \hp_{S/ I_2(M_{t,1})}(z)(z(1-z)^{n(t-1)})+z(1-z)^{(m-1)n}+(1-z)^{m},
    \end{equation*}
    where $I_2(M_{t,1})$ as defined in Notation \ref{nota.i_t1}.
\end{theorem}
\begin{proof}
     Let $<$ be the monomial order introduced in Remark \ref{rem.gr}. From Remark \ref{rem.gr}, one has $\ini_<(I(\mcC))=\alpha + \sum_{r=1}^{m-1} \beta_r$, where $\alpha= (x_{j1}x_{ik} \mid 1 \leq i<j \leq m, 2 \leq k \leq n)$ and $\beta_r=(x_{r1}x_{il}x_{kj} \mid r \leq i<k \leq m, 2 \leq j<l \leq n)$. Set $x_{01}= 0 \in S$. Consider the following sequence of short exact sequence 
    \begin{equation} \label{eq:inik}
    \begin{split}
        0 \longrightarrow \frac{S}{\ini_<(I(\mcC))+(x_{01},\ldots,x_{t-1,1}):x_{t1}}(-1) &\longrightarrow \frac{S}{\ini_<(I(\mcC))+(x_{01},\ldots,x_{t-1,1})} \\
     &\longrightarrow \frac{S}{\ini_<(I(\mcC))+(x_{11},\ldots,x_{t1})}   \longrightarrow  0,
    \end{split}
    \end{equation}
where $t=1,\ldots,m-1$. Since the Hilbert-Poincar\'e polynomial is additive over short exact sequences, computing the Hilbert-Poincar\'e polynomial of first and third modules in Equation (\ref{eq:inik}) one yields the desired formula.

First, observe that, for $t=2, \ldots, m$ one has 
\begin{equation*}
\begin{split}
    \ini_<(I(\mcC))+(x_{11},\ldots, x_{t-1,1})=&(x_{11},\ldots,x_{t-1,1}) + (x_{j1}x_{ik} \mid j\geq t, 1 \leq i < j \leq m, 2 \leq k \leq n) \\
    &+ (x_{r1}x_{il}x_{kj} \mid r \geq t, r \leq i<k \leq m, 2 \leq j<l \leq n).
\end{split}
\end{equation*} Since  $x_{j1}x_{ik},x_{r1}x_{il}x_{kj} \in (x_{11},\ldots,{x_{t-1,1}})$ for all $j \leq t-1$ and $r \leq t-1$, the above equality follows. Notice that when $t=m$, the ideal $\ini_<I(\mcC)$ corresponds to the edge ideal of a star graph, as described below: $$\ini_<(I(\mcC))+(x_{i1} \mid 1 \leq i \leq m-1)=(x_{i1} \mid 1 \leq i \leq m-1) + (x_{m1}x_{ik} \mid  1 \leq i \leq m-1, 2 \leq k \leq n).$$ 
Then from Remark \ref{rem.starhp} it follows that $\hp_{S/\ini_<(I(\mcC))+(x_{i1} \mid 1 \leq i \leq m-1)} (z)=z(1-z)^{(m-1)n}+(1-z)^m$. Next, for $t=1, \ldots, m-1$ one has 
    \begin{equation*}
    \begin{split}
    \ini_<(I(\mcC))+(x_{01},\ldots ,x_{t-1,1}):x_{t1}=&(x_{i1} \mid 1 \leq i \leq t-1)+ (x_{ik} \mid 1 \leq i < t , 2 \leq k \leq n)\\
    &+(x_{j1}x_{ik} \mid t \leq i < j \leq m, 2 \leq k \leq n)\\
    &+ (x_{il}x_{kj} \mid t \leq i<k \leq m, 2 \leq j<l \leq n) \\
    =& (x_{ik} \mid 1 \leq i \leq t-1 , 1 \leq k \leq n) +\ini_<(I_2(M_{t,1})).
    \end{split}
    \end{equation*}
    Since the minimal generating set $I_{t,1}$ is a Gr\"obner basis of $I_2(M_{t,1})$ (see Remark \ref{rem.Mgb}), the second equality follows. The first equality follows from the fact that $\{u/\gcd(u,x_{t1}) \mid u \in \text{mingens}(I)\}$ is the  generating set of $I:x_{t1}$, where $I=\ini_<(I(\mcC))+(x_{01},\ldots ,x_{t-1,1})$. Then we have that $\hp_{S/\ini_<(I(\mcC))+(x_{01},\ldots ,x_{t-1,1}):x_{t1}} (z)= \hp_{S/ \ini_<(I_2(M_{t,1}))}(z)((1-z)^{n(t-1)})$. Now, by substituting the values into short exact sequence (\ref{eq:inik}) for $t=m-1, \ldots,1$,  we obtain the Hilbert-Poincar\'e series of $\hp_{S/\ini_<(I(\mcC))}(z)$. Since the initial ideals are compatible with Hilbert-poincar\'e polynomial the assertion follows.
\end{proof}

\begin{remark}
    In \cite{QRR2022}, the authors computed the Hilbert–Poincar\'e series of the ideal generated by all $2$-minors of the matrix $M$ (see also \cite{BGS1982}). Therefore, by substituting the Hilbert–Poincar\'e polynomial of $I_2(M_{t,1})$ into Theorem \ref{thm.hpp}, we obtain the Hilbert–Poincar\'e polynomial of the ideal $I(\mcC)$, where $\mcC$ is the set of all corner-interval minors of $M$.
\end{remark}

We denote by $H=\{x_{11},x_{12},\ldots,x_{1n}\}$ the first row and $V=\{x_{11},x_{21},\ldots,x_{m1}\}$ the first column of variables, and by $(H)$ and $(V)$ the ideal generated by the elements of $H$ and $V$, respectively.

\begin{theorem} \label{thm.regc}
    Let $\mcC$ be the set of all corner minors of $M$ and $m,n>2$. Then, one has $\reg(S/I(\mcC))=\min(m,n)$.
\end{theorem}

To complete the proof of Theorem \ref{thm.regc}, we need the following results.

\begin{lemma} \label{lemma.reg,i11}
    Let $I_2(M)$ be the ideal generated by all $2$-minors of $M$ and $m,n>2$ (see, Notation \ref{nota.i_t1}). Then $\reg(S/(I_2(M)\cap (H) \cap (V)))=\min(m,n)$.
\end{lemma}
\begin{proof}
    First, we claim that $I_2(M)\cap (H) \cap (V)=I(\mcC)+x_{11}I_2(M_{2,2})$, where $\mcC$ is the set of all corner minors of $M$. It is clear that $I(\mcC)+x_{11}I_2(M_{2,2}) \subset I_2(M)\cap (H) \cap (V)$. 
    
    Set $J=I(\mcC)+x_{11}I_2(M_{2,2})$. For the other inclusion, first we claim that $J:x_{11}= I_2(M)$. Since $I_2(M)$ is prime and $x_{11} \notin I_2(M)$, we have $J:x_{11} \subseteq I_2(M)$. For reverse inclusion, since  $x_{11}(x_{1c}x_{bd}-x_{1d}x_{bc}) =x_{1c}(x_{11}x_{bd}-x_{b1}x_{1d})-x_{1d}(x_{11}x_{bc}-x_{b1}x_{1c})$, and $x_{11}(x_{a1}x_{bd}-x_{ad}x_{b1}) =x_{a1}(x_{11}x_{bd}-x_{b1}x_{1d})-x_{b1}(x_{11}x_{ad}-x_{a1}x_{1d})$, it follows that $x_{11}f_\delta \in I(\mcC)$ for every $\delta$ involving the first row or first column. For the remaining $\delta$, $x_{11}f_\delta \in x_{11}I_2(M_{2,2})$. Hence, $x_{11}I_2(M) \subseteq J$, proves the claim. Moreover, $$J+(x_{11}) =(x_{11})+(x_{i1}x_{1j}:i,j\geq 2) =(H)\cap(V).$$  
    
    Let $g\in I_2(M)\cap(H)\cap(V)$. Since $g\in J+(x_{11})$, we may write $g=h+sx_{11}$, where $h\in J$ and $s\in S$. As $g,h\in I_2(M)$, we have $x_{11}s\in I_2(M)$. Since $I_2(M)$ is prime and $x_{11}\notin I_2(M)$, it follows that $s\in I_2(M)=J:x_{11}$. Thus $x_{11}s\in J$, and hence $g\in J$, as desired.

    From \cite[Proposition 3.3.]{K2020} one has $\reg(S/I_2(M))=\min\{m-1,n-1\}$.  Note that $J + (x_{11})=(x_{11})+(x_{i1}x_{1j} \mid 1<i\leq m, 1<j\leq n)$ and the ideal $(x_{i1}x_{1j} \mid 1<i\leq m, 1<j\leq n)$ is the edge ideal of the complete bipartite graph on vertex set $H\setminus x_{11} \sqcup V \setminus x_{11}$. Since it is the edge ideal of a cochordal graph, it follows from Remark \ref{rem.froberg} that $\reg(S/J + (x_{11}))=1$. Consider the following short exact sequence \begin{equation} \label{eq:I:11}
    0 \longrightarrow \frac{S}{J:x_{11}}(-1) \longrightarrow \frac{S}{J}  \longrightarrow \frac{S}{J+(x_{11})}   \longrightarrow  0.
    \end{equation}
    Since $m,n>2$, by applying Remark \ref{rem.Reg} to short exact sequence (\ref{eq:I:11}) we get  $\reg(S/J)=\min(m,n)$.
\end{proof}

\begin{remark} \label{rem.reg,i11}
    From \cite[Equation 3.6]{DES1998} it follow that $\ini_<(I(\mcC)+(H)^2+(V)^2)=((H)+(V))^2$ with respect to the term order defined in Remark \ref{rem.gr}. After polarization, the ideal $((H)+(V))^2$ corresponds to the edge ideal of a graph obtained by attaching a whisker to each vertex of a complete graph on the vertex set $\{H,V\}$. Since it is a cochordal graph, it follows from Remark \ref{rem.froberg} that $\reg(S/\ini_<(I(\mcC)+(H)^2+(V)^2))=1$. Moreover, $\reg(S/(I(\mcC)+(H)^2+(V)^2))\leq \reg(S/\ini_<(I(\mcC)+(H)^2+(V)^2))=1$.
\end{remark}

\begin{proof}[Proof of Theorem \ref{thm.regc}]
    Form \cite[Theorem 3.2]{DES1998} it follows that $I(\mcC)=I_2(M_{1,1})\cap (H) \cap (V)\cap (I(\mcC)+(H)^2+(V)^2)$, where $I_2(M_{1,1})$ as defined in Notation \ref{nota.i_t1}. Consider the following short exact sequence 
    \begin{equation} \label{eq:IJ}
    0 \longrightarrow \frac{S}{I(\mcC)=I\cap J} \longrightarrow \frac{S}{I} \oplus \frac{S}{J}
    \longrightarrow \frac{S}{I+J}   \longrightarrow  0,
    \end{equation}
    where $I=I_2(M_{1,1})\cap (H) \cap (V)$ and $J=I(\mcC)+(H)^2+(V)^2$. Form Lemma \ref{lemma.reg,i11} and Remark \ref{rem.reg,i11}, it follows that the regularity of the middle module in Equation (\ref{eq:IJ}) is equal to $\min\{m,n\}$. To compute the regularity of the third module in short exact sequence (\ref{eq:IJ}). We first show that $I(\mcC)+x_{11}I_2(M_{2,2})+(H)^2+(V)^2:x_{11}=(H)+(V)+I_2(M_{2,2})$. It is clear that  $(H)+(V)+I_2(M_{2,2}) \subset (I(\mcC)+x_{11}I_2(M_{2,2})+(H)^2+(V)^2:x_{11})$. For other inclusion, first we establish that the minimal generating set of $I(\mcC)+x_{11}I_2(M_{2,2})+(H)^2+(V)^2$ form a Gr\"obner basis with respect to term order defined in Remark \ref{rem.gr}. 
    From Remark \ref{rem.Mgb} it follows that the minimal generating set of $(H)+(V)+I_2(M_{2,2})$ form a Gr\"obner basis. Then, using [Proposition 15.12]\cite{E1995} we get that $\ini_<(I(\mcC)+x_{11}I_2(M_{2,2})+(H)^2+(V)^2:x_{11})=\ini_<((H)+(V)+I_2(M_{2,2}))$, and thus are equal. Thus, it is reduce to show that $S(f,g)$ reduce to $0$ for any $f,g$ in the minimal generating set of $I(\mcC)+x_{11}I_2(M_{2,2})+(H)^2+(V)^2$. It is known that $S(f,g)$ reduce to $0$ if $\ini_<f$ and $\ini_<g$ are coprime or $f$ and $g$ are monomials. In the following we will show the remaining cases:
    \begin{enumerate}
        \item $S(x_{i1}x_{1j}-x_{11}x_{ij},x_{i1}x_{1j'}-x_{11}x_{ij'}) = x_{11}x_{ij'}x_{1j}-x_{11}x_{ij}x_{1j'} \in H^2$, for all $2\le i \le m$, and  $2\le j,j' \leq n$ with $j \ne j'$;
        \item $S(x_{i1}x_{1j}-x_{11}x_{ij},x_{i1}x_{j'1}) = x_{11}x_{ij}x_{j'1} \in V^2$, for all $2\le i ,j'\le m$, and  $2\le j \leq n$.
    \end{enumerate}
    Other pairs can be treated similarly. Note that the $S$-pair $S(x_{11}(x_{il}x_{kj}-x_{kl}x_{ij}),x_{11}(x_{i'l'}x_{k'j'}-x_{k'l'}x_{i'j'}))$ reduce to $0$, since the minimal generating set $x_{11}I_2(M_{2,2})$ form a Gr\"obner basis (see Remark \ref{rem.Mgb}). Likewise, the $S$-pair $S(x_{i1}x_{1j}-x_{11}x_{ij},x_{11}(x_{i'l'}x_{k'j'}-x_{k'l'}x_{i'j'}))$ reduce to $0$, because the initial terms are coprime, as desired.

    Then, from \cite[Proposition 3.3]{K2020} it follows that  $\reg(S/((H)+(V)+I_2(M_{2,2})))=\reg(S/(I(\mcC)+x_{11}I_2(M_{2,2})+(H)^2+(V)^2):x_{11})=\min\{m-2,n-2\}$.

    Since $I(C)+x_{11}=(H\setminus x_{11})(V\setminus x_{11})$, it is easy to verify that  $I(\mcC)+x_{11}I_2(M_{2,2})+(H)^2+(V)^2 + (x_{11})=(x_{11})+(H\setminus x_{11}+V\setminus x_{11})^2$. 
    %Since  $I(\mcC)+x_{11}I_2(M_{2,2})+H^2+V^2 + x_{11}= (x_{i1}x_{1j} \mid 2 \leq i \leq m, 2 \leq j \leq n)+(H\setminus x_{11})^2+(V\setminus x_{11})^2 + x_{11}$, the equality follows.  
    Polarization of  the ideal $(H\setminus x_{11}+V\setminus x_{11})^2$ is the edge ideal of  graph obtained by attaching a whisker to each vertex of a complete graph on vertex set $\{H\setminus \{x_{11}\},V\setminus \{x_{11}\}\}$. Then from Remark \ref{rem.froberg} it follows that  $\reg(S/(x_{11}+(H\setminus x_{11}+V\setminus x_{11})^2))=1$. Applying Remark \ref{rem.Reg} to the following short exact sequence 
    \begin{equation*} 
    \begin{split}
       0 \longrightarrow \frac{S}{I(\mcC)+x_{11}I_2(M_{2,2})+(H)^2+(V)^2:x_{11}}(-1) \longrightarrow \frac{S}{I(\mcC)+x_{11}I_2(M_{2,2})+(H)^2+(V)^2} \\
    \longrightarrow \frac{S}{I(\mcC)+x_{11}I_2(M_{2,2})+(H)^2+(V)^2+(x_{11})}   \longrightarrow  0, 
    \end{split}
    \end{equation*}
    yields $\reg(S/I(\mcC)+x_{11}I_2(M_{2,2})+(H)^2+(V)^2)=\min\{m-1,n-1\}$. Since we obtain the regularity of the middle and the last modules of Equation (\ref{eq:IJ}), again applying Remark \ref{rem.Reg} to Equation (\ref{eq:IJ}) one yields $\reg(S/I(\mcC))=\min\{m,n\}$, as desired.
\end{proof}

In the following proposition, we obtain a bound for the regularity of powers of the ideal generated by binomials that corresponds to an arbitrary collection of corner-interval minors in terms of its subcorner-interval minors.

\begin{proposition}
    Let $\mcC$ be an arbitrary set of corner-interval minors and let $V_1,\ldots,V_r$ be the maximal vertical interval of $\mcC$. If $\mcC'$ be a subcorner-interval minors obtained by removing $V_k$ from $\mcC$, for some $k$, then 
    $$\beta_{ij}(S/I(\mcC')^t) \leq \beta_{ij}(S/I(\mcC)^t)$$
    for all $i,j \geq 0$ and $t\geq 1$.
\end{proposition}
\begin{proof}
    Let $\mcC'$ be a subcorner-interval minors of $\mcC$ by removing a maximal vertical interval $V_k$ and $S_{\mcC'}=K[x_{ij} \mid x_{ij} \in V(\mcC')]$ be a polynomial ring. First, we claim that $I(\mcC')^t=I(\mcC)^t \cap S_{\mcC'}$, for all $t \geq 1$, where $I(\mcC')$ is the binomial ideal of $\mcC'$ in the ring $S_{\mcC'}$. Since generators of the ideal $I(\mcC')^t$ contained in the ideal $I(\mcC)^t$, it follows that $I(\mcC')^t \subseteq I(\mcC)^t \cap S_{\mcC'}$. For the other inclusion, consider the map $\psi: S \rightarrow S_{\mcC'}$ by setting $\psi(x_{ij})=0$ if $x_{ij} \notin V(\mcC')$ and $\psi(x_{ij})=x_{ij}$ if $x_{ij} \in V(\mcC')$. Let $g=\sum_{\delta_{1},\ldots,\delta_t \in \mcC}a_{\delta_{1},\ldots,\delta_t}f_{\delta_1}\cdots f_{\delta_t}$ be an element in $I(\mcC)^t$, where $a_{\delta_{1},\ldots,\delta_t} \in S$. Then, we have 
    \begin{equation*}
    \begin{split}
        \psi(g) &=\sum_{\delta_{1},\ldots,\delta_t \in \mcC}\psi(a_{\delta_{1},\ldots,\delta_t})\psi(f_{\delta_1}\cdots f_{\delta_t})\\
        &=\sum_{\delta_{1},\ldots,\delta_t \in \mcC'}\psi(a_{\delta_{1},\ldots,\delta_t})f_{\delta_1}\cdots f_{\delta_t} \in I(\mcC')^t.
    \end{split}
    \end{equation*}
    Therefore, we have  $I(\mcC)^t\cap S_{\mcC'} \subseteq I(\mcC')^t$. Then the statement follows from \cite[Corollary 2.5]{OHH2000} that if $S_{\mcC'}/I(\mcC')^t$ is an algebra retract of ${S}/{I(\mcC)^t}$.
      Consider, ${S_{\mcC'}}/{I(\mcC')^t} \xhookrightarrow{i} {S}/{I(\mcC)^t} \xrightarrow{\bar{\psi}} {S_{\mcC'}}/{I(\mcC')^t}$, where $\bar{\psi}$ is an induced by the map $\phi$. Then one can see that $\bar{\psi} \circ i$ is identity on ${S_{\mcC'}}/{I(\mcC')^t}$ as desired.
\end{proof}

\begin{remark}
    The above result also holds when maximal horizontal intervals are considered in place of maximal vertical intervals. In particular, if the resulting collection $\mcC'$ satisfies the hypotheses of Theorem \ref{thm.regc}, then $\reg(S/I(\mcC')) \leq \reg(S/I(\mcC))$, provides a lower bound for $\reg(S/I(\mcC))$.
\end{remark}

\vspace{2mm}

\noindent {\bf Acknowledgement.}
The author is supported by the Scientific and Technological Research Council of Turkey T\"UB\.{I}TAK under the Grant No: 124F113, and thankful to T\"UB\.{I}TAK for their supports. The author would like to thank Ayesha Asloob Qureshi for
the helpful discussions.

\vspace{2mm}

\noindent {\bf Data availability statement.}
This article has no associated data.

\end{document}